\documentclass[12pt,a4paper]{article}
\usepackage{amsmath}
\usepackage{amstext}
\usepackage{amsfonts}
\usepackage{amssymb}
\usepackage{amsthm}
\usepackage{graphicx}
\usepackage{hyperref,amssymb}
\usepackage{xcolor}
\newtheorem{theorem}{Theorem}[section]

\newtheorem{examples}[theorem]{Examples}
\newtheorem{definition}[theorem]{Definition}
\newtheorem{definitions}[theorem]{Definitions}

\newtheorem{lemma}[theorem]{Lemma}

\newtheorem{proposition}[theorem]{Proposition}
\newtheorem{remark}[theorem]{Remark}

\newcommand{\de}{\delta}
\newcommand{\si}{\sigma}

\newcommand{\ut}{\underline{t}}
\newcommand{\ua}{\underline{a}}
\newcommand{\ub}{\underline{b}}
\newcommand{\ud}{\underline{\delta}}
\begin{document}

	\title{\textbf{Structure and arithmetic of multivariate Ore extensions}}
	
	\author{Andr\'e Leroy$^{a}$ \thanks{andre.leroy@univ-artois.fr} ~~~ Huda Merdach$^{b}$ \thanks{huda\_merdach@du.edu.eg} \\
		$^{a}$Artois Universit\'e, UR 2462 (LML), Facult\'e of Sciences,\\ 62 300, Lens, France. 
		\\ $^{b}$ Department of Mathematics, Faculty of Science,\\
		Damietta University, Damietta 34517, Egypt.}
	\date{}
	\maketitle	

\textbf{Abstract:}  We give the basic structure of the multivariable Ore extensions $S=A[\ut; \sigma, \ud]$ introduced in the work of Mart\'i{n}ez-Pe\~nas and Kschischang. The Pseudo multilinear transformations (PMT's) are introduced and correspond to modules over $S$.  These maps are strongly connected to the evaluation of polynomials in $S$. A general product formula is obtained. PMT's help to put some structure on the set of roots  of a polynomial $f(t) \in S$. 
	
	\textit{\textbf{keywords}:}  Pseudo multilinear transformation, Skew Polynomial ring, Centralizer, Semi-invariant, Roots of polynomials.
	\section{Introduction}
	 
		~~~ The evaluation of polynomials is at the heart of many areas of mathematics. The Ore extension rings (or skew polynomial rings) are one of the most engaging  notions of polynomials in noncommutative algebra. The first appearance  of Ore extension $K[t; \si, \delta]$ dates back to Ore (cf. \cite{Or}) in 1933. Numerous authors studied skew polynomials and their evaluations in particular when the coefficient ring is a division ring or a prime ring (cf. e.g. \cite{L2012}).  Ore extensions have been used in ring theory as a source of examples (cf. e.g. \cite{CP 2003}, \cite{L2012}) they also give useful tools in quantum groups \cite{BG}. Furthermore, they appeared more recently in coding theory (cf. e.g. \cite{BL2013}, \cite{BU cod 2009}, \cite{BGU}).   
		
	This paper is concerned with a construction of a noncommutative polynomial ring, denoted $S=A[\ut;\sigma,\ud]$, that is essentially due to U. Mart\'i{n}ez-Pe\~nas and F. R.
	Kschischang (cf. \cite{MP2019},\cite{MP2022}).  This construction already appeared in an unnoticed paper by 
	Elena Kreindler \cite{K}.   The $n$ variables $t_1,t_2,\dots,t_n$ are free variables and this extension $S$  has a different behavior than the "usual" iterated Ore extension (cf. \cite{BL2023}).  We slightly extend the context by  considering a general ring $A$ for the coefficients of the polynomials.  

	 In Section 2, some basic properties and examples are given.
		  We introduce the PMT.   These maps are 
		  our main tool.  The 
		  use of PMT allows a study of both the 
		  left $S$-modules and their morphisms.   
		  This generalizes previous works that 
		  appear in case of one variable (cf. \cite{L2012}, \cite{L 1995}).  This is given in Proposition (\ref{Pro for varphi and 1 1 
		  correspondence}).  The PMT's also play a fundamental role in the evaluation of an 
	  element $f(\ut)\in S$. We give several 
	  examples in (\ref{examples CV and others}) and  (\ref{examples of PMT}).  One of the main results in this section is a complete description of the left $S=A[\ut; \sigma, \ud]$-modules and their morphisms (see in particular, Proposition (\ref{Pro ring of endo})). 
	  
	  In Section 3, we  determine the center of $S$ when the base ring is a division ring, in Proposition (\ref{Pro for 
	  	the center of $S$}). Also, we introduce the semi-invariant polynomials and construct several examples in  (\ref{example for semi-invariant polyomial}) and (\ref{examples for Semi invariant and others}). 
	In Theorem (\ref{Th multivariate Ore extension}) we give, under some hypothesis, the structure of semi-invariant polynomials. 
%	 constructintroduce when $S_i = K[t_i, \si_i, \delta_i]$  that contain in $S= K[\ut, \si, \ud]$ is not simple and the case that all monic semi-invariant polynomials contained in $S_i$.
	 
	  In section 4, the evaluation of polynomials is presented. This is completely different from the evaluation in iterated Ore extensions defined in \cite{BL2023}. In addition, we study the relations between evaluation and PMT in Proposition (\ref{Pro for the link between evaluation and PMT}). 
	  In particular, we obtain a general product formula in Proposition (\ref{Pro product formula}) even when the base ring is not a division ring. We define a relation $\sim$ between elements in $A^n$. In Proposition (\ref{Pro Roots and conjugation}) and Proposition (\ref{Pro semi invarent and right multiplication}), PMT's are used to describe the decomposition of the set $V(f)=\{\ua \in A^n \mid f(\ua)=0\}$,$f(\ut)\in S=A[\ut,\si,\ud]$, into its $ \sim $ classes. 
	  %For $f(\ut)\in S=A[\ut,\si,\ud]$, the decomposition of the set of its roots
	 %$V(f)=\{\ua \in A^n \mid f(\ua)=0\}$ in classes define by $\sim$ can described via PMT's in Proposition (\ref{Pro Roots and conjugation}) and Proposition (\ref{Pro semi invarent and right multiplication}).   
	
	   In the last section, we introduce ($\sigma,\ud$)-centralizer.  We give different characterizations of these sets in Proposition (\ref{Pro for centralizer}).  To each element $\ua \in A^n$ we attach, in Proposition (\ref{Pro T_a is right C linear}), a PMT $T_{\ua}$ and show that $T_{\ua} $ is right linear over the ($\sigma,\ud$)-centralizer of $\ua$ . Finally,  for a domain A, and an element  $f \in S = A[\ut; \si, \ud] $, we describe the set of roots of a polynomial $V(f(\ut))$ in terms of the kernel of $f(T_{\ua})$.  The main result for this section is Proposition (\ref{Pro for domain}) that presents some structure on the set of roots of polynomial $V(f(\ut))$.

	All the rings will be associative with identity. 
	
	\section{Structure of multivariate Ore extensions}
	
	~~~In this section, we introduce our main objects and the tools that we will use.  In particular, the Pseudo Multivariate Transformations are defined and applications of these maps are given in (cf. \cite{MP2019}).
	
	\begin{definition}
		Consider a ring $A$, $n$ variables $t_1,\dots,t_n$, $\si: A \rightarrow M_n (A)$ a ring homomorphism, and a sequence of $n$ additive maps $\de_1,\dots,\de_n$.  We denote by 	$M$ the free monoid generated by the variables $\{t_1,\dots,t_n\}$
		and by $S=A[\ut; \si,\ud]$ the set of polynomials of the form $\sum_{m\in M}\alpha_mm$, where $\alpha_m\in A$ and $m\in M$.  On this set, we define the natural addition and we introduce a multiplication based on the concatenation in $M$ and on the following commutation rules:
		\begin{equation}
		\forall \; 1\le i \le n, \;\forall \; a\in A, \quad 
		t_ia= \sum_{j=1}^n \si(a)_{ij}t_j +\de_i(a).  
	\end{equation}
	\end{definition}
	
	For editorial reasons, for $a\in A$, we will write $\si_{ij}(a)$ instead of  $\si(a)_{ij}$, viewing $\si_{ij}$ as a map from $A$ to $A$.  The next proposition gives some key features of this construction.  We leave the proof to the reader.
	
	\begin{proposition} \label{Pro for assoc. and hom}
		(1) The associativity of the ring $S$ leads to the following rule for 
		the maps $\de_1,\dots,\de_n$:
		\begin{equation}
		\forall a,b \in A, \quad \de_i(ab)=
		\sum_{j=1}^n\si_{ij}(a)\de_j(b)+\de_i(a)b.
		\end{equation}
		In a compact form, this can be written as $\underline{\de}(ab)=\si(a)\underline{\de}(b)+\underline{\de}(a)b.$  The sequence of maps $\delta_{\ua}$ will be called a $\sigma$-derivation.
		
		(2) The fact that $\si$ and $\ud$ satisfy  the above properties can also be summarized by asking that the map $\phi$ from $A$ to the  matrix ring $M_{(n+1)\times (n+1)}(A)$ defined by
		$$
		\phi:\; A \rightarrow M_{(n+1)}(A) \; \text{with} \; a \mapsto \begin{pmatrix}
			\si(a) & \underline{\de}(a) \\
			0 & a
		\end{pmatrix} \text{is a ring homomorphism.}$$ 
	\end{proposition}

	\begin{examples}
		\label{examples CV and others}
		{\rm
			\begin{enumerate}
			\item Let $\ua=(a_1,\dots,a_n)^t\in A^n$.  We define \\ 
			$\delta_{\ua}(x)= \ua x-\sigma(x)\ua$ in other words, 
			$\delta_{\ua}=(\delta_{a_1},\delta_{a_2},\dots, \delta_{a_n})^t$ where 
			$\delta_{a_i}(x)=a_ix-\sum_{j=1}^n\sigma_{ij}(x)a_j$. 
			One can check that $\delta_{\ua}$ is 
			indeed a $\sigma$-derivation.   When 
			$\delta=\delta_{\ua}$, we can erase the 
			derivation in the sense that, 
			$A[\ut;\sigma,\delta_{\ua}]=A[\ut-\ua;\sigma]$.  Such a derivation will be called it multivariate inner derivations.
%			or $$A[\ut;\sigma,\delta_{\ua}]=A[\delta_{\ua};\ua - \si].$$
%			
			\item Similarly if there exist $U\in Gl_n(A)$ and $\tau_1,\dots,\tau_n$ automorphisms of the ring $A$, such that, for every $x\in A$, we have 
			$$
			\sigma(x)=U(diag(\tau_1(x),\dots \tau_n(x)))U^{-1}
			$$
			then, noting $\tau=diag(\tau_1,\dots,\tau_n$) and
			%$$
			%A[\ut;\sigma,\ud]=A[U\utU^{-1}
			%$$
			$\underline{y}=U^{-1}\ut$, we get, for any $x\in A$, $\underline{y}x= U^{-1}\ut x=U^{-1}(\sigma(x)\ut+\ud(x))=U^{-1}\sigma(x)\ut+U^{-1}\ud(x)=U^{-1}\sigma(x)UU^{-1}
			\ut+U^{-1}\ud(x)=\tau(x)\underline{y}+U^{-1}\ud(x)$. One can check that $U^{-1}\ud(x)$ is a $\tau$-derivation, so that we can write
			$$
			A[\ut;\sigma,\ud]=A[\underline{y};\tau,U^{-1}\ud].
			$$
			
			\item  Assume that $A=K$ is a division ring finite-dimensional over its center $k$ and that $\sigma(\alpha)=diag(\alpha,\dots,\alpha)\in M_n (K)$ for any $\alpha\in k$, then by a direct application of the Skolem Noether theorem (cf. Cohn, P. M. Book \cite{C}, p. 262) we obtain that there exists an invertible matrix $U\in Gl_n(K)$ such that
			 $\sigma(a)=Udiag(a,\dots,a)U^{-1}$ for every $a\in K$.  In particular, using the previous item we get that
			$$
			K[\ut,\sigma, \ud]=K[\underline{y};Id.,U^{-1}\ud]
			$$
			where $\underline{y}=U^{-1}\ut$.
			\item  If $\sigma$ is diagonal, in other words if $\sigma=diag(\sigma_1,\dots, \sigma_n)$ then, for any $1\le i \le n$, the commutation rules are $t_ia=\sigma_i(a)t_i +\delta_i(a)$, where $\delta_i$ is a $\sigma_i$-derivation.  In this case, the Ore extension $A[\ut;\sigma,\ud]$ contains all the one variable Ore extensions $A[t_i;\sigma_i,\delta_i]$.
			
			\item  \label{End R} Let $A$ be a ring, $\alpha,\beta\in End(A)$, and $\gamma$ be an ($\alpha,\beta$)-derivation (i.e. $\gamma \in End(A,+)$ and, for any $a,b\in A$ we have $\gamma(ab)=\alpha(a)\gamma(b)+ \gamma(a)\beta(b)$).  We can check that the map  
			$$
			\sigma: A \longrightarrow M_2(A):\; a \mapsto
			\begin{pmatrix}
				\alpha(a) & \gamma(a) \\
				0 & \beta(a)
			\end{pmatrix} 
			$$
			is a homomorphism of rings.  If $x\in A$ we can define an  ($\alpha,\beta$)-derivation $\gamma$ via 
			$\gamma(a)=x\beta(a)-\alpha(a)x$.  Such an ($\alpha,\beta$)-derivation is called inner.  For more information on ($\alpha,\beta$)-derivations the reader may consult \cite{BCM}.  The map $\sigma$ above gives rise to the extension $A[(t_1,t_2)^t;\sigma]$.
			
			\item Let us notice that in the case of an upper triangular $\sigma$ of the form 
			$$
			\sigma(a)=\begin{pmatrix}
				\alpha(a) & \delta(a) \\
				0 & a
			\end{pmatrix}
			$$
			We get that $\delta:A\longrightarrow A$ is an $\alpha$-derivation and we can
			 consider both $R=A[t;\alpha, \delta]$ and $S=A[\ut; \sigma]$ where $\ut=\begin{pmatrix}
				t_1 \\
				t_2
			\end{pmatrix}$.
		Let us remark that the map $ \varphi : S\rightarrow R$ defined by $\varphi (t_1)=t, \, \varphi(t_2)=1$ and $\varphi (a)=a $ for all $a\in A$ is a ring homomorphism between $S$ and $R$. 
		
		\item We can generalize the points (5) and (6) above 
		as follows.  Let $A$ be a ring, $\alpha: A 
		\longrightarrow M_n(A)$, and $\beta: A 
		\longrightarrow M_l(A)$ be morphisms of rings.  A 
		map $\gamma: A \longrightarrow M_{n\times l}(A)$ is 
		an $(\alpha, \beta)$-derivation if $\gamma$ is 
		additive and satisfies 
		$\gamma(ab)=\alpha(a)\gamma(b)+ \gamma(a)\beta(b)$.  As above, this leads to
		$$
		\sigma: A \longrightarrow M_{n\times l}(A):\; a \mapsto
		\begin{pmatrix}
			\alpha(a) & \gamma(a) \\
			0 & \beta(a)
		\end{pmatrix} 
		$$
    and we get a multivariable extension with $n+l$ variables
    $A[(t_1,\dots, t_{n+l})^t;\sigma]$.  As a special case, 
    we can consider an inner $(\alpha,\beta)$-derivation via 
    a matrix $x\in M_{n \times l}(A)$ and define, for $a\in 
    A$, $\gamma(a)=x\beta(a)-\alpha(a)x$.
    We leave to the reader the analogue of (6).
   \end{enumerate}	
} 
\end{examples}
	
	\vspace{4mm}
	
	We now introduce the important notion of PMT.
 We keep our usual notation $S=A[\ut;\sigma,\ud]$.
	If $V$ is a left $S$-module,  then $V$ is also a left $A$-module and, for any $1\le i \le n$, the action of $t_i$  on $V$ must satisfy the following equality 
	\begin{equation} 
		(\ut a).v=(\sigma(a)\ut + \ud (a)).v ~~~~~~~ for~~ v \in V.  
	\end{equation}
	%Put~ simple~ formulas~
	This leads to the next definition.
	
	\begin{definition}
	\label{Definition of PMT}
		Let $V$ be a left $A$-module and
		$T_1,\dots,T_n \in \mathrm{End}(V,+)$ be such that, for $a\in A$ and $v\in V$, we define 
		\begin{equation} \label{eq 1}
			 T_i(a.v)=\sum_{j=1}^n \si_{ij}(a)T_j(v) + \de_i(a).v.  ~~\forall \; 1\le i \le n.
		\end{equation}

		A sequence of maps satisfying these equations will be called a 
		$(\si, \underline{\de})$-pseudo-multilinear transformation $((\si, \underline{\de})$-PMT, for short$)$ on $V$. 
\end{definition}

\vspace{3mm}

In other words, writing $T=(T_1,T_2,\dots, T_n)^t$ for a column of elements in $\mathrm{End}(V,+)$, we can write the equality in (cf. equation \ref{eq 1}) in a compact form as follows:
$$
T(a.v)=\si(a)T(v)+\underline{\de}(a)v.
$$
	
	\begin{examples}
		\label{examples of PMT}
		{\rm 
			\begin{itemize}
			
			\item[(a)] If we consider $\sigma =Id.$ and $\delta=0$ the maps $T_1, T_2,\dots , T_n$ are just usual linear maps on $V$.   In general our PMT form a sequence of maps.   Let us remark that these maps are, in general, not linear.
			\item[(b)] One can check that the sequence $\underline{\de}=(\de_1,\dots,\de_n)^t$ is a  PMT on $A$.  
%			What we just said is that there is a  one-to-one  correspondence between the left modules over $S$ and the set of PMTs over the left $A$-modules.
			
			\item[(c)]  Let $\ua=(a_1,\dots,a_n)^t$ be a column $\in A^n$ then the 
			PMT on $A$ defined as follows $T_{\ua}=(T_{a_1},\dots,T_{a_n})^t$ with 
		\begin{equation} \label{T_{a_i}(b)}
			T_{a_i}(b)=\sum_{j=1}^n\sigma_{ij}(b)a_j + \de_i(b).
		\end{equation}
	We can check that we indeed get a PMT defined over $A$. As we will see, this PMT is closely related to the evaluation at $\ua$.  This example already appeared in \cite{MP2022}, Definition 5.
		\item[(d)] Let us remark that if we consider $ \ua = (0, \dots, 0)^t \in A^n$, then the PMT $T_{\ua}$ is simply the map $(\delta_1, \dots, \delta_n)^t $.  
		%i.e $$T_{a_i}(b) = \de_i(b).$$
		
	\end{itemize}
				
		}
	\end{examples}

\vspace{4mm}
	
	As in the case of a single variable, we can associate a ring homomorphism to each PMT.  This is the purpose of the next proposition.
	
	\begin{proposition} \label{Pro for varphi and 1 1 correspondence}
		Let $T$ be a PMT defined on left $S$-module $V$. Then
		\item (1) The following map
			$$
		\varphi: S \rightarrow \mathrm{End}(V,+)~~\text{such that}~~      \varphi(f(\ut))=f(\underline{T}), 
		$$
		is a ring homomorphism.
		
		\item (2)	There is a 1-1 correspondence between the set of PMT's and the set of $S$-modules.
	\end{proposition}
	\begin{proof}
		(1)  The map $\varphi$ is additive and we only need to check that it is also multiplicative.
		We have, for every $a\in A$ and $1\le i \le n$, $T_iL_a=\varphi(t_ia)=\varphi(\sum_j\sigma_
		{ij}(a)t_j+\delta_i(a))=\sum_{j}\sigma_{ij}(a)T_j +L_{\delta_i(a)}$.
		
		(2) If $T=(T_1,\dots,T_n)$ is a PMT on a module $_AV$ we obtain a left $S=A[\ut,\sigma,\ud]$-module structure on $V$ by defining $t_i.v=T_i(v)$.  On the other hand when $_SV$ is a left $S$-module the actions of $t_1,\dots,t_n$ on $V$ give a PMT on $V$ as in the paragraph before the definition \ref{Definition of PMT}.
	\end{proof}
	If $_SV$ is
	a left $S$-module such that $_AV$ is free of dimension $l$ and if $B$ is a basis of $V$, the actions of $t_1,\dots,t_n$ on $V$ are completely described by $n$ matrices $\{\tau_1,\dots,\tau_n\}\subset M_l(A)$ expressing these action on the basis. These matrices are sufficient to describe the left $S$-module structure  of $V$.  
	% Explicitly if $v=\sum_{i=1}^{i=l}a_iv_i$ is the expression of $v\in V$ in the basis $B=\{v_1,\dots,v_l\}$ then
	%$$
	%t_iv=t_i(\sum \alpha_jb_j)=\tau_i
	%$$
	Suppose that $V_1$ and $V_2$ are two left $S=A[\ut;\sigma,\ud]$-modules such that both
	$_AV_1$ and $_AV_2$ are free with basis  
	$\beta_1=\{e_1,\dots,e_{n_1}\}$ and $\beta_2=\{u_1,\dots,u_{n_2}\}$ respectively.  We denote the matrices corresponding to these actions in the respective basis by $X_1,\dots,X_n\in M_{n_1\times n_1}(A)$ and $Y_1,\dots, \,Y_n \in M_{n_2\times n_2}(A)$. If $V_1\stackrel{\varphi}{\rightarrow} V_2 $ is a left $A$-morphism, we let $M\in M_{n_2 \times n_1}(A)$ to be the matrix representing $\varphi$ in the basis $B_1$ and $B_2$.
	\paragraph{}	
	Now, suppose that $S=A[\ut; \sigma, \ud]$  is a multivariate Ore extension.
	For $i=1,2$, let $T_i=(T_{i1},\dots,T_{in})^t$ be $(\si,\de)$-PMT defined on $V_1$ and $V_2$, respectively. If $\varphi \in Hom_A(V_1,V_2)$ is an $A$-module homomorphism, then  $M\in M_{n_1\times n_2}(A)$, $X=(X_1,\dots,X_n)\in M_{n_1\times n_1}(A)$ and $Y=(Y_1,\dots,Y_n) \in M_{n_2\times n_2}(A) $ denote 
	matrices representing $\varphi$, $T_1$ and $T_2$ 
	respectively in the appropriate basis $\beta_1$ and $\beta_2$.
	Let $_SV_1$ and $_SV_2$ be the left $R$-module structures corresponding to $T_1$ and $T_2$, 
	respectively. We have the following theorem:

	\begin{theorem} 	\label{Pro ring of endo}
		The following conditions are equivalent: 
		\begin{enumerate}
			\item[(i)] $\varphi \in Hom_S(V_1,V_2)$;
			\item[(ii)]  $\varphi T_{1i}=T_{2i} \varphi$, for every $1\le i \le n$;
			\item[(iii)] $X_iM=\sum_j\sigma_{ij}(M)Y_j+ \delta_i(M)$ for every $1\le i \le n$.
			%		\item[(iv)] $B\in \ker(T_{C_2}-L_{C_1})$ where $T_{C_2}$ (resp. 
			%		$L_{C_1}$) stands for the pseudo-linear transformation 
			%		(resp. the left multiplication) induced by $C_2$ (resp. $C_1$) on 
			%		$M_{n_1\times n_2}(A)$ considered as a left $M_{n_1}(A)$-module.
		\end{enumerate}
	\end{theorem}
	\begin{proof}
		%	(ii)$\Leftrightarrow$ (iii) It is clear that the matrix 
		%	corresponding to $\varphi T_{1,i}$ is the 
		%	$n_2 \times n_1$ matrix $X_iM$.  Let 
		%	$\beta_1=\{v_1,\dots, v_{n_1}\}$ and $\beta_2=\{w_1, \dots, w_{n_2}\}$.
		%	For $1\le k \le n_1$ and $1\le l \le n_2$, we have 
		%	$$
		%	(X_iM)_{l,k}=(\varphi(T_{1,i}(v_k)))_l=(T_{2i}\varphi(v_k))_l=T_{2i}(\sum_jM_{jk}w_j)_l=$$
		%	$$ \sum_j	\sum_s\sigma_{is}(M_{jk})T_{2s}(w_j)+\delta_i(w_j)=$$
		Firstly, we have 
		\begin{align*}
			(X_iM)_{(l,k)} &=(\sum_{js}\sigma_{is}(M_{jk}) T_{is}(w_{j}) + \delta_{i}(M_{jk}) w_{j})_{l} \\& = (\sum_{j,s}\sigma((M_{ik})_{js}) \sum_{p=1}^{n_{2}}(Y_{s})_{pj}w_{p} + \delta_{i}(M_{jk})w_{j})_{l} 
			\\&= \sum_{j,s} \sigma((M_{jk})_{is}) (Y_{s})_{lj} + \delta_{i}(M_{lk})
			\\&= \sum_{s}(\sum_{j} \sigma_{is}(M_{k})(Y_{s})_{lj}) + \delta_{i}(M_{lk}) 
			\\&= ((\sum_{s} \sigma_{is}(M) Y_{s}) + \delta_{i}(M))_{lk} 
			 = \sum_{s} (\sum_{j} \sigma_{is}(M_{lj})(Y_{s})_{jk})
		\end{align*}
		$ \text{Also},~ (\varphi\circ T_{i1})_{lk}  = (\varphi(T_{i1}(v_{k})))_{l} = \varphi(\sum_{j=1}^{n_{2}}(X_{i})_{jk}v_{j})_{l}
			= (\sum_{j=1}^{n_{2}}(X_{i})_{jk} \varphi(v_{j}))_{l} 
		 = (\sum_{j=1}^{n_{2}}(X_{i})_{jk} (\sum_{s}M_{sj}w_{s}))_{l} 
			 = \sum_{j=1}^{n_{2}}(X_{i})_{jk} M_{lj}. $\\
	$ \text{Now  (i) $\Leftrightarrow$ (ii)}~
			\varphi (t_{i} . v_{j}) = t_{i}. \varphi(v_{j})  \Leftrightarrow \varphi(T_{1i}(v_{j})) = T_{i2}(\varphi(v_{j}))   \Leftrightarrow (\varphi \circ T_{1i}(v_{j})  = (T_{i2} \circ \varphi)(v_{j}).$ 
		\\(ii)$\Leftrightarrow$ (iii) 
		$M(\varphi \circ T_{1i})_{lk} = (\varphi \circ T_{1i})(v_{l})_{k} = (M(T_{1i})M(\varphi))_{lk} = (X_{i}M)_{lk}$
		
		On the other hand,\begin{align*}  M(\varphi \circ T_{1i})_{lk} &= M(T_{2i} \circ \varphi)_{lk}  = \sum_{k} ((T_{2i} \circ \varphi)(v_{l}))_{k})w_{k} 
		\\&	= \sum_{k} (T_{2i}(\varphi(v_{l})))_{k}w_{k} 
			 =  (T_{2i}(\sum_{j}M_{lj}w_{j}))_{k} 
			\\ &= (\sum_{j}\sum_{s} \sigma_{is} (M_{lj})T_{2i}(w_{j})+ \delta_{i}(M_{lj})w_{j})_{k} 
			\\ &= (\sum_{j}\sum_{s} \sigma_{is} (M_{lj}) \sum_{m}((Y_{i}))_{jm}(w_{m})+ \delta_{i}(M_{lj})w_{j})_{k}
			\\ & =  (\sum_{j}\sum_{s}\sum_{m} \sigma_{is}(M_{lj})(Y_{i})_{jm}w_{m} + \sum_{j} \delta_{i}(M_{lj})w_{j})_{k} 
			\\ &= \sum_{j,s}\sigma_{is}(M_{lj})(Y_{i})_{jk} + \delta_{i}(M_{lk}) 
			 = \sum_{j,s} \sigma_{is}(M)_{lj}(Y_{i})_{jk} + \delta_{i}(M_{lk}) 
			\\ & = \sum_{s}(\sigma_{is}(M)Y_{i})_{lk} + \delta_{i}(M)_{lk} 
			 = (\sum_{s}\sigma_{is}(M)Y_{i} + \delta_{i}(M))_{lk} . 
		\end{align*}
	\end{proof}
\vspace{5mm}

A classical feature of one variable Ore extensions is the fact that $R=K[t;\sigma,\delta]$ is embeddable in a division ring when $K$ is itself a division ring.
Since $R$ is a left principal domain, this is immediate.   Although in our more general setting $S=K[\ut,\sigma,\ud]$ is not even Noetherian, it is also embeddable in a division ring.  We will not use the following theorem and hence mention it with a sketch of proof. 
\begin{theorem}
	Let $K$ be a division ring and $S=K[\ut,\sigma,\ud]$.  Then $S$ is embeddable in a division ring. 
\end{theorem}
\begin{proof}
	We first show that the ring $S$ is filtered via the length of monomials.  Moreover, this filtration satisfies the weak algorithm and hence is a fir (cf. Section 2.4, in particular Theorem 2.4.4 and Theorem 2.4.6 in \cite{PMC}).  We conclude that $S$ is indeed embeddable in a division ring (cf. Corollary 7.5.14 in \cite{PMC}). 
\end{proof}
	\section{Center of S and Semi-invariant polynomials}
%	~~~In this section we introduce the center for 
	 ~~~The purpose of the next proposition is to study the center of S so, we consider K a division ring and $ S = K[\ut, \sigma, \ud] $ where $ \ut = (t_{1}, ..., t_{n}) , \sigma = (\sigma_{ij})$ and $ n > 1. $ Then, 
	\begin{proposition}
		\label{Pro for the center of $S$}
		 The center Z(S) of S is 
		\begin{center}
			 $Z(K)_{(\sigma, \ud)} = \{a \in K \mid ab = ba~ \forall~ b \in K ; \si(a) = a. I_{n} , \delta_i (a) = 0, ~ \forall \; 1\le i \le n\}$  
		\end{center} 
	\begin{proof} 
	%Let $P(t) = \sum_{\omega \in \varOmega} a_{\omega} , \omega \in Z(S) $, where $\varOmega $ is the semigroup generated by $ t_{1}, ..., t_{n}.$ We order $ \varOmega$ by the deg lex order with $ t_{1} <...< t_{n} $ 
%			\\Let $ \alpha\omega $ be the leading term of $ P(t)$ $( \alpha \in K^{*}, \omega \in \varOmega)$. Since $ P(t) $ is central and deg lex is a term order. $\forall ~i \in \{1, ..., n\}$ 
%			$$ t_{i} \alpha\omega = \alpha\omega t_{i} , ~~~~ \forall ~ 1 \le i  \le n $$
%			$$ \sum_j \delta_{ij}(\alpha) t_{j}\omega + \delta_i(\alpha) \omega = \alpha\omega t_{i} $$
%			Comparing leading terms we have 
%			$$ \delta_{in} (\alpha) t_{n}\omega = \alpha \omega t_{i} ~~~ \forall~ 1 \le i \le n $$
%			$$ \alpha^{-1} \sigma_{in}(\alpha) t_{n} \omega = \alpha \omega t_{i} $$
%			From this we conclude $ \omega = 1 $
%			$$ \alpha^{-1} \sigma_{in}(\alpha) t_{n} = t_{i}  ~~~~~ \forall \; 1\le i \le n $$
%			$ ; \sigma_{in}(\alpha) = 0 , ~\forall~ \sigma \in \{1, ..., n-1\} ~ and ~ \sigma_{nn} (\alpha) = \alpha $ 
Let $P(t) = a_s m_s + ... + a_0m_0 \in Z(R)$ with $m_0 = 1$.  For all $0 \leq i \leq n$, we have $t_i P(t) = P(t) t_i. $ The number of monomials on the right hand side is $ s + 1 $ so, the same must be true on the left hand side. For any $ 1 \leq j \leq s, $ we have 
\begin{equation} \label{eq 1}
	t_i a_j m_j = \sum_{k=1}^{n} \sigma_{ik}(a_j) t_k m_j. 
\end{equation} 
Since $\sigma(a_j)$ is an invertible matrix, we have for, $(i,j) \in \{1,\dots, n\} \times \{ 0,\dots, s\}$, ($\sigma_{i1}(a_j),\dots ,\sigma_{in}(a_j))\ne (0,\dots,0)$.
Since $R m_j \cap  Rm_i=0, ~ i \neq j $, the non zero terms of the sum in \ref{eq 1} will not cancel with terms coming from $t_i a_l m_l $ for any $l \neq j,~ 0 \leq l \leq s$. Since the number of non zero monomic on the left hand side is $s + 1$ we can concluded that the  sum in (\ref{eq 1}) has only one term: 
$t_i a_j m_j = \sigma_{ik} (a_j) t_k m_j. $ for some $1 \leq k \leq n.$  

This terms must be equal to one of the non constant terms in $P(t) t_i $. i.e. we must have $\sigma_{ik}(a_j) t_km_j = a_r m_r t_i , ~ 1 \leq r \leq s$ and hence $ t_km_j = m_r t_i $. 
So for every $1\le i \le n$ and for every $1\le j \le s$, we have that $t_i$ divides $m_j$ on the right.  Since the variables $t_1, \dots, t_n$ are independant this impossible and we conclude that $ P(t) = a_0 \in K \Rightarrow P(t) \in K$ , hence $P(t) \in Z(K).$ With $p(t) = a \in Z(K), ~ t_i a = a t_i \Rightarrow \sigma(a) = a I.$ and we get that $P(t) \in Z(K)_{(\sigma, \delta)}$ as described that fact that $Z(K)_{(\sigma, \delta)} \subset Z(R)$ is obvious.

			Now  $ P(t) \in Z(S) \Longrightarrow P(t) \in K $
			So, $ P(t) = \alpha \in Z(S)$ 
			, $ \forall ~i$ we have $ t_{i} \alpha = \alpha t_{i} \Longrightarrow \sum_j \sigma_{ij} (\alpha) t_{j} + \delta_i (\alpha) = \alpha t_i$ 
			\\ $ \forall ~i, j \in \{1, ..., n\}, \sigma_{ij} (\alpha) = 0 ~ if ~ i \neq j $
			; $ \sigma (\alpha) = \left( \begin{array}{ccc}
				\alpha	& ...  &0 \\
				&\ddots& \\
				0	& ... & \alpha
			\end{array}\right) , \sigma_{ii}(\alpha) = \alpha $
			\textit{Moreover}, $ \delta_i (\alpha) = 0 ~~ \forall~ i$ and $\alpha a = a \alpha ~~~ \forall~ a \in K \Longrightarrow \alpha \in Z(K)_{(\sigma, \ud)} .$
		\end{proof}
	\end{proposition}
	
	\begin{definition}
		A nonzero polynomial $p(\ut)\in S$ is right semi-invariant if for any $a\in K$ there exists an $a'$ in $K$ such that $p(\ut)a=a'p(\ut)$.
	\end{definition}
	
	\begin{lemma}
		\label{semi-invariant poly and its map}
		Suppose that $p(\ut)\in S$ is right semi-invariant.  Then there exists a homomorphism $ \varphi $ from $K$ to $K$ such that $p(\ut)a=\varphi(a)p(\ut)$.
	\end{lemma}
	\begin{proof}
		Let us notice that for $a\in K$, there exists a unique element $a'\in K$ such that  $p(t)a=a'p(t)$.  Since the element $a'$ is unique we can define the map $\varphi:K\rightarrow K$ such that $\varphi(a)=a'$.  It is easy to check that $\varphi$ is a ring homomorphism.  
	\end{proof}
	
	\begin{examples} \label{example for semi-invariant polyomial}
		{\rm 
			(1)  Let $K$ be a division ring, and consider a map $\sigma=diag(\sigma_1,\sigma_2)$ and $\delta=(\delta_1,\delta_2)=(0,0)$.   Assume that $\sigma_1^l=\sigma_2^l$, then one can check that $t_1^l+t_2^l$ is a semi-invariant polynomial in $S=K[\ut,\sigma,\ud]$.
			
			(2) Let $K$ be a division ring of characteristic $2$, and consider maps $\sigma=diag(Id,Id)$ and $\delta=(\delta_1,\delta_2)$, where be two usual derivations on $K$ are such that $\delta_1^2=\delta_2^2=0$, then $t_1^2+t_2^2$ is a semi-invariant polynomial in $S=K[\ut,\sigma,\ud]$. 
		}
	\end{examples}

	Let us remark that $a'$ is unique.  
	In the case when $n=1$ these semi-invariant polynomials are at the heart of the structure theory since such a nonconstant semi-invariant polynomial exists if and only if the Ore extension is not simple.
	In our general frame, the semi-invariant notion is too rigid to give any structure result.  Nevertheless in some particular cases, these polynomials exist and their zeroes behave nicely.  We will analyze this behavior in the next section and now we will just construct these polynomials.  
	In the case when $\sigma$ is diagonal, say $\sigma = diag (\sigma_1,\dots, \sigma_n)$
	we can search the semi-invariant polynomials in the subrings $K[t_i,\sigma_i, \delta_i]$, where $1\le i \le n$. 
	
	\vspace{3mm} 
	\begin{theorem}\label{Th multivariate Ore extension}
		Let $S=K[\ut,\sigma,\ud]$ be a multivariate Ore extension such that there exists $1\le i \le n$ with $\sigma=diag (\sigma_1,\dots, \sigma_n)$ where $\sigma_i\in Aut(K)$.  Then the skew polynomial ring $S_i=K|t_i,\sigma_i,\delta_i]$ is contained in $S$.
		We assume that there exists a nonconstant semi-invariant polynomial $p_i(t_i)\in S_i$.  Then 
		\begin{enumerate}
			\item For $1\le i \le n$, the ring $S_i$ is not simple if and only if there exists a monic nonconstant semi-invariant polynomial of minimal non zero degree.
			\item Suppose that $p_i(t_i)$ is as in (1) then all the monic semi-invariant polynomials contained in $S_i$ are of the form $\sum_{j=0}^l a_jp_i(t_i)^j$ for some $a_j \in K$ with $a_l=1$.   
		\end{enumerate}
	\end{theorem}
	\begin{proof}
		These results are extracted from (cf. \cite{LM}).
	\end{proof}
	
	\begin{examples} \label{examples for Semi invariant and others}
	{\rm 
		\begin{enumerate}
			\item if $\sigma=diag(\sigma_1,\sigma_2,\dots, \sigma_n)$ and $\delta_i=0$ then, for any $a\in A$, $t_ia=\sigma_i(a)t_i$.  This shows that $t_i$ is semi-invariant (even invariant).
			\item If there exists $1\le i \le n$ such that for every $1\le j \le n$ we have $\sigma_{ij}=\sigma_{i}\delta_{ij}$ (where $\delta_{ij}$ stands for the classical Kroeneker symbol) and $\delta_i$ is quasi algebraic (cf. \cite{LM})  then there exists a monic invariant polynomial $p(t_i)$, say of degree $l$, such that $p(\delta_i)(x)=\sigma_i^l(x)p(t_i)$ so that the polynomial $p(t_i)$ is semi-invariant.  
			\item Let $\alpha,\beta,\gamma$ be as in Examples (cf. \ref{examples CV and others}) part number (\ref{End R})). Suppose that $\alpha\gamma=-\gamma\beta$ then one can check that $t_1^2$ and $t_2^2$ are semi-invariant polynomials in $A[(t_1,t_2)^t,\sigma]$.
	
	\item Let us now give an example of a multivariate Ore extension $S$ that is simple.   This will be very similar to the Weyl algebra construction.
	We start with the field of rational fractions $k(x)$ over a field $k$ of characteristic zero and define $\sigma :k(x) \rightarrow M_2(k(x))$ and $\delta_1=\delta_2 $ via
	$$
	\sigma(p(x))=\begin{pmatrix}
		p(x) & 0 \\
		0 & p(x)
	\end{pmatrix} \quad {\rm and} \quad \begin{pmatrix}
		\delta_1(p(x)) \\
		\delta_2(p(x))
	\end{pmatrix} =\begin{pmatrix}
		p'(x) \\
		p'(x)
	\end{pmatrix}
	$$
	We will show that $S=k(x)[\begin{pmatrix}
		t_1 \\
		t_2
	\end{pmatrix},\sigma,\begin{pmatrix} 
		\delta_1 \\
		\delta_2
	\end{pmatrix}]$ is simple.
	define the usual deglex order on the monomials in the variables $t_1,t_2$.   Assume that $I$ is a nonzero two-sided ideal of $S$ and let $f=f(t_1,t_2)\in I$ be nonzero polynomial with minimal deglex order amongst nonzero elements of $I$.   If $f\in k(X)$ we get that $f$ is invertible and hence $I=S$.  So let $w\ne 1$ be the deglex leading monomial in $f$.  An easy computation shows that the deglex order of $xw-wx$ is smaller than that of $w$.  Hence the the deglex order  of $fx-fx\in I$ is smaller than that of $f$.  This implies that $fx=xf$, and hence the same is true for the leading term of $f$.  This implies that $f\in k(x)$ a contradiction.

	\item Let us notice that in the previous example when the characteristic of $k$  is finite,  the ring $S$ will not be simple anymore.  For instance if $char(k)=2$, we have that the left ideal generated by 
$I=St_1^2 + St_2^2 +\sum_{w\in \Omega}St_1^2t_2w +\sum_{w\in \omega} 
t_2^2t_1w$ is a two sided ideal of $S$.  It is easy to check that $\delta_1^2=\delta_2^2=0$ and this implies that the elements $t_1^2$ and $t_2^2$ are in the kernel of the ring homomorphism (see Proposition (\ref{Pro for varphi and 1 1 correspondence})) $\varphi: S \rightarrow End(k(x),+)$ which is  associated to the MLT defined by the point $(0,0)$.	
\end{enumerate}
}
\end{examples}

	\section{Evaluation and $(\si,\ud)$-conjugation}
	
~~~The evaluation of skew polynomials in one variables were in particular studied in  \cite{LM}, \cite{L2012}, and \cite{BL2023}. In the setting of multivariate extensions over a division ring they were defined in Mart\'i{n}ez-Pe\~nas and Kschischang's paper  \cite{MP2019}.   Here we extend the definitions to the case on a general coefficient ring.	
	\begin{definitions}
		\label{Evaluation and conjugation}
		\begin{enumerate}
			\item We define the evaluation of $f(\ut)\in S=A[\ut; \si,\underline{\de}]$ at $(a_1,\dots,a_n)\in A^n$, via the representative of $f(\ut)+I\in S/I$ by an element of $A$, where $I$ is the left ideal $I=S(t_1-a_1)+S(t_2-a_2)+\cdots +S(t_n-a_n)$. 
			\item 	If $x\in U(A)$ we denote $\ua^x$ the $(\sigma,\ud)$-conjugate of $\ua$ (a column in $A^n$) by $x$ defined by
			\begin{equation}
			\ua^x=\sigma(x)\ua x^{-1}+\ud(x)x^{-1}
		\end{equation}
			\item For $\ua, \ub \in A^n$ we define 
			$\ua\sim \ub$ if there exists a nonzero 
			divisor $x\in A$ such that $\ub 
			x=\sigma(x)\ua +\ud(x)$.   We put
			\begin{equation} \label{a symetric to b}
			\Delta (\ua)=\{\ub \in A^n \mid \ua\sim \ub \}.
		\end{equation}
		\end{enumerate}
	\end{definitions}

%\begin{remark}
{\rm 
	It is important to remark that for a general ring $A$, the relation in  (cf. equation (\ref{a symetric to b}))  is not symmetric and hence doesn't lead to an equivalence relation. 
}
%\end{remark}

	\begin{examples} \label{examples for evaluation}
		{\rm 
			\begin{enumerate}
				\item If we suppose $n=2$, then evaluating $t_1t_2$ at $(a_1, a_2)$ we get 
				$(t_1t_2)(a_1,a_2)=\si_{11}(a_2)a_1+\si_{12}(a_2)a_2+\de_1(a_
				2)$.  Let us now compare with $t_2t_1$ evaluated at $(a_1,a_2)$.  We also have $(t_2t_1)(a_1,a_2)=\sigma_{22}(a_1)a_2+\sigma_{21}(a_1)a_1+\delta_2(a_1)$.  
				
				\item When $\sigma=(\sigma_1,\dots,\sigma_n)$ is diagonal we have, for $1\le i \le n$ and $a\in A$, $t_ia=\sigma_i(a)+ \delta_i(a)$ and hence, the skew polynomial rings $A[t_i,\sigma_i,\delta_i]$ are contained in $S$.  We compute $(t_1t_2)(a_1,a_2)=\sigma_1(a_2)a_1+\delta_1(a_2)$ and $(t_2t_1)(a_1,a_2)=\sigma_2(a_1)a_2+\delta_2(a_1)$.
				%		\item ****Evaluation when the base ring is finite.
				
			%	\item If n = 3 then evaluating $t_1t_2t_3$ at $(a_1,a_2, a_3)$ we get 
			%	$t_1t_2t_3(a_1,a_2, a_3) = $ 
			\end{enumerate}
		}	
	\end{examples}
	 Let us remark that the evaluations that we obtain in the above examples are very different from the evaluations that appear when considering iterated extensions (cf. \cite{BL2023}). 
	\vspace{3mm}
	
	Since $S/I$ is a left $S$-module,  it gives rise to a  $(\si, \underline{\de})$-PMT on $S/I$ given by the actions of $t_i$ for $1\le i \le n$.  The elements of $S/I$ are represented by a unique element 
	of $A$ so  that the action of $t_i$ on $S/I$  can be described by
	$$t_i.(x+I)=t_ix+I=\sum     \si_{ij}(x)a_j+\de_i(x) +I.$$ 
	
	The PMT attached to this action is 
	$T_{\ua}=(T_{a_1},T_{a_2},\dots, T_{a_n})$ where, for $x\in A$ and $1\le i \le n$, we have $T_{a_i}(x)=\sum_{j=1}^n\si_{ij}(x)a_j + \de_i(x)$ (cf. Examples  \ref{examples of PMT} equation number (\ref{T_{a_i}(b)})).
	
	  Let us recall from Proposition (\ref{Pro for varphi and 1 1 correspondence}) that the map $\varphi_{\ua}:S=A[\ut,\sigma,\ud] \rightarrow End(A,+)$ defined by $\varphi_{\ua}(f(t_1,\dots,t_n))=f(T_{a_1},\dots, T_{a_n})$ is a ring homomorphism.
	
	\vspace{3mm}

	The next proposition established a link between evaluation and PMT. This was given in Theorem 1 of \cite{MP2022} for the case when the base ring is a division ring.  The case of a general ring as a base ring was also proved for univariate extensions in the papers \cite{L 1995} and \cite{L2012}. 
	
	\begin{proposition}\label{Pro for the link between evaluation and PMT}
		For $f(\ut)\in S=A[\ut,\sigma,\ud]$ and $\ua\in A^n$ we have
		$$    
		f(\ua)=f(T_{\ua})(1).
		$$
	\end{proposition}
	\begin{proof}
		Since $f(\ut)$ is a sum of monomials, it is enough to prove this formula for a monomial.  Let 
		$w=t_{i_1}t_{i_2}\cdots t_{i_l}$ be such a 
		monomial.  We proceed by induction on the length of 
		$w$.  If this length is one, we have $w=t_{i_1}$ 
		for some $1\le i_1 \le n$.  Since $\sigma(1)= I_n$, 
		we have that $T_{a_i}(1)=a_{i_1}$.  Hence, 
		$t_{i_1}(\ua)=a_{i_1}=T_{a_i}(1)$.  
		
		Assume that the formula is true for monomials of length $l$, for some $l\ge 1$, and consider a monomial of length $l+1$: $w=w't_i$ where $w'$ is of length $l$.  We then have $w(\ua)=(w't_i)(\ua)=
		(w'(t_i-a_i)+w'a_i)(\ua)$ and since $w'(t_i-a_i)\in S(t_i-a_i)$ we have  $w(\ua)=(w'a_i)(\ua)$.  Using the induction hypothesis we obtain $w(\ua)= 
		(w'a_i)(T_{\ua})(1)=(\varphi_{\ua}(w'a_i))(1)=(\varphi_{\ua}(w')\circ\varphi_{\ua}(a_i))(1)=
		w'(T_{\ua})(a_i)=
		w'(T_{\ua})((T_{a_i})(1))=w'(T_{\ua})(t_i(T_{\ua})(1))=(w't_i)(T_{\ua})(1)=w(T_{\ua})(1)$. 
	\end{proof}		
	\vspace{3mm}
	   The fact that the map $\varphi$ in Proposition (\ref{Pro ring of endo}) is a ring 
	homomorphism,  then immediately leads to part (1) of the following proposition.  This classical formula is called the ``product formula".  In the case of one variable it can be found in \cite{L2012}.  Part (2) of the next proposition was extended to the multivariable setting over a division ring in  \cite{MP2019}.   Here we give it for a general coefficient ring.
	\begin{proposition}
		\label{Pro product formula}
		Suppose that $f,g \in S$, $\ua\in A^n$, and $x\in A$.
		\begin{enumerate}
			\item We have:
			$$(fg)(\ua)=f(T_{\ua})(g(\ua))$$   In particular,  if 
			$g(\ut)=x\in A$, then we have $(f\circ x)(\ua)=f(T_{\ua})(x)$.
			\item Assume that $0\ne  g(a) \in U(A) $, then we get:
			$$
			(fg)(\ua)= f(\ua^{g(\ua)})g(\ua).
			$$
		\end{enumerate}
		\begin{proof}
			(1) $fg(\ua)=fg(T_{\ua})(1)=
			(\varphi(fg))(1)=(\varphi(f)\circ\varphi(g))(1)=
			\varphi(f)(\varphi(g)(1))\\
			f(T_{\ua})(g(T_{\ua})(1))=f(T_{\ua})(g(\ua))$, where $\varphi$ is the map associated to the PMT $T_{\ua}$, as defined in Proposition (\ref{Pro for varphi and 1 1 correspondence}). 
			
			(2)   We put $x=g(\ua)$ and $I=\sum S(t_i-a_i^x)$.   We have that, $f-f(\ua^x)\in \sum S(t_i-a_i^x)$.  Since  $(\ut -\ua^x)x=\sigma(x)(\ut-\ua)$, we get that $fx-f(\ua^x)x\in \sum S(t_i-a_i^x)x\in \sum S(t_i-a_i)$.  This shows that  $(fx)(\ua)=f(\ua^x)x$ and hence $(fg)(\ua)=f(T_{\ua})(g(\ua))=f(T_{\ua})(x)=(fx)(\ua)=f(\ua^x)x$. 
		\end{proof}		
	\end{proposition} 
	
	The first equality in the previous proposition shows how the use of $T_{\ua}$ leads to a general product formula for polynomials with coefficients in a general base ring $A$.  This also gives a 
	link between the kernel of $f(T_{\ua})$ and the roots of $f(\ut)$.  In case $A=K$ is a division ring, and $f.x \neq 0 $, the fact that  $f(T_{\ua})(x)=(f.x)(\ua)=f(\ua^x)x$ shows that the kernel of $f(T_{\ua})$ corresponds to roots of $f(\ut)$.  The same is true for a domain but requires some formalism. 
	%*** the reader will find them in the next section.
	
	\begin{proposition}
		\label{Pro Roots and conjugation}
		Let 
		$\ua,\ub\in A^n$  be such that there exists a nonzero divisor $x\in A$ with $\ub x= \sigma(x)\ua + \ud(x)$.  Then
		\begin{enumerate}
			\item For any $y\in A$, $f(T_{\ub})(y)x=f(T_{\ua})(yx)$
			\item We have 
			$x\in \ker f(T_{\ua}) $ if and only if $f(\ub)=0$.  In particular, if $x\in U(A)$, $x\in ker(f(T_{\ua}))$if and only if $f(\ua^x)=0$.
		\end{enumerate}
	\end{proposition}
	\begin{proof}
		(1)  Let us first compute 
		$T_{\ub}(y)x=\sigma(y)\ub x + 
		\ud(y)x=\sigma(y)\sigma(x)\ua 
		+\sigma(y)\ud(x)+\ud(y)x=\sigma(yx)\ua + \ud(yx)=T_{\ua}(yx)$.
		We use an induction on the length $l$ of a word $w=w(t_1,\dots , t_n)$ to prove the formula for monomials.  If $l=1$, $w(t_1,\dots,t_n)=t_i$ for some $1\le i \le n$ and the desired equality is just the $i^{th}$ row of the formula $T_{\ub}(y)x=T_{\ua}(yx)$, that we just proved.  Now assume the formula has been proved for a word $w=w(t_1,\dots,t_n)$ and let us show it holds for $w(t_1,\dots,t_n)t_i$ where $1\le i \le n$.  We have $(w(t_1,\dots,t_n)t_i)(T_{\ub})(y)x=
		w(T_{b_1},\dots ,T_{b_n})(T_{b_i}(y))x$.   Thanks to the induction hypothesis we obtain that 
		$w(T_{b_1},\dots ,T_{b_n})(T_{b_i}(y))x=w(T_{a_1},\dots,T_{a_n})((T_{b_i})(y)x)$.  Using the formula obtained for $l=1$,  leads to $w(T_{a_1},\dots,T_{a_n})((T_{b_i})(y)x)=
		w(T_{a_1},\dots,T_{a_n})(T_{a_i})(yx)$ and we conclude that $(w(t_1,\dots,t_n)t_i)(T_{\ub})(y)x=w(T_{a_1},\dots,T_{a_n})(T_{a_i})(yx)$, as desired. 
		The fact that the formula is true for a polynomial is now obvious.
		
		(2) Considering the equation in (1) with $y=1$, we get $f(\ub)x=f(T_{\ub}(1))x=f(T_{\ua})(x)$ and the fact that $x$ is not a zero divisor immediately gives that $x\in \ker f(T_{\ua})$ if and only if $f(\ub)=0$.  The last assertion is clear.
	\end{proof}
	
	We first give a consequence of Proposition (\ref{Pro product formula}) on the roots of a semi-invariant polynomial.   We write $ V(f) = \{\ua \in K^n ; f(\ua) = 0\} \subset K^n$ for the set of roots of $f\in S$.
	
	\begin{theorem} \label{Th Dring be semiinvar}
		Let $p(\ut)\in S=K[\ut; \sigma, \ud]$, where $K$ is a division ring be a semi-invariant polynomial.  Then for any $\ub\in V(p)$  we have that $\Delta^{\sigma,\delta}(\ub)\subset V(p)$.
	\end{theorem}
	\begin{proof}
		By hypothesis, for every 
		$a\in K\setminus\{0\}$, we have that $p(\ut) a=\varphi(a)p(\ut)$.  Hence we have 
		$(p(\ut)a)({\underline b})=p(\ut)({\underline b}^a)a=
		(\varphi(a)p(\ut))(\ub)= \varphi(a)p(\ub)=0$, by the product formula.  This shows that for any $a\in K\setminus\{0\}$, we have that
		$p(\ut)({\underline b}^a)=0$, as required. 
	\end{proof}
	
	%\begin{remarks}
	%{\rm 
		%		(1) Since the map $\varphi$ associated with a  PMT is a ring homomorphism from $S$ to $\mathrm{End}(V,+)$, when $\varphi$ is not injective, the multivariate polynomial ring is not simple.   The simplicity is thus related to some algebraicity of a PMT, exactly as in the case of a single variable.  We will not go deeper into this subject.
		%%		
		%		MR4394033 (sent May 2022)*** Kim, Nam Kyun et al., Annihilating
		%		properties of ideals generated by coefficients of polynomials
		%		and power series. Internat. J. Algebra Comput. 32 (2022), no.
		%		2
		
		%		2) As a final remark let us mention that, if the division ring $K$ is finite-dimensional over its center $F$ and $\sigma$ is $F$-linear, the Skolem Noether theorem shows that $\sigma$ is diagonalizable.  In other words, there exist an invertible matrix $U$ and a set of $n$ automorphisms of $K$, say $\sigma_1, \dots, \sigma_n$, such that $\sigma = Inn_U \circ diag(\sigma_1,\dots,\sigma_n) $.  In this situation, the multivariate extension $S=K[\ut; \si,\underline{\de}]$ contains the ore extensions $S_i=K[t_i;\sigma_i,\delta_i]$.  In the iterated Ore extension this last fact is always true.
		%}
	%\end{remarks}
	
	\begin{proposition} \label{Pro semi invarent and right multiplication}
				Suppose that $\sigma=diag(\sigma_1,\dots, \sigma_n)$ and that, for some $1\le i \le n$ there is no polynomial $q(t_i)\in S_i=K[t_i,\sigma_i,\delta_i]$ such that $q(a)=0$ for every $a\in K$.  Then $p_i(t_i)\in S_i$ is semi-invariant if and only if, for every $a\in K$,  
		$$
		p_i(T_a)=r_{p_i(T_a)(1)}\circ \sigma_i^{n_i}
		$$
		where $n_i=deg (p_i(t_i))$ and $r_{p_i(T_a)(1)}$ stands for the right multiplication by $p_i(T_a)(1)$.
	\end{proposition}
	\begin{proof}
		Using Proposition(\ref{Pro product formula}), we have, for any $x\in K$, $p_i(T_a)(x)= 
		(p_i(t_i)x)(a)=(\sigma_i^{n_i}(x)p_i(t_i))(a)
		=\sigma_i^{n_i}(x)(p_i(T_a)(1))=r_{p_i(T_a)(1)}\circ \sigma_i^{n_i}(x)$.  This gives the formula.
		
		Conversely, if $	p_i(T_a)=r_{p_i(T_a)(1)}\circ \sigma_i^{n_i}$ then, for every $x\in K$, $p_i(T_a)(x)=\sigma_i^{n_i}(x)p_i(T_a)(1)$.
		Thus for every $x,a\in K$, $(p_i(t_i)x)(a)= (\sigma_i^{n_i}(x)p_i(t_i))(a)$.  Hence our hypothesis shows that, for any $x\in K$, $p_i(t_i)x=\sigma_i^{n_i}(x)p_i(t_i)$, showing that $p_i(t_i)$ is semi-invariant.
	\end{proof}
	
	%\vspace{5mm}
	
	\section{Centralizers and roots}
	
	~~~In this section we study the important role played by the centralizer, Also, we will assume $A$ is a (noncommutative) domain and $S$ will stand for $S=A[\ut;\sigma,\ud]$.
	
	%\vspace{3mm}
	
	\begin{definitions}
		Let $\ua=(a_1,\dots,a_n)^t\in A^n$, the $(\sigma,\ud)$-centralizer of $\ua$, denoted $C^{(\sigma,\ud)}(\ua)$ is the set
		\begin{equation} \label{centralizer}
			C^{(\sigma,\ud)}(\ua)=\{x\in A \mid \sigma(x)\ua + \ud(x)=\ua x \} \subset A
		\end{equation}

		The idealizer, denoted $\rm{idl}(I)$ of a left ideal $I$ of $S=A[\ut,\sigma, \ud]$ is defined by 
		$\rm{idl}(I)=\{p\in S\mid Ip\subset I\}$.
	\end{definitions}
	
	One can easily check that $C^{(\sigma,\ud)}(\ua)$ and $\rm{idl}(I)$ are in fact  subrings of $A$ and $S$ respectively.  Moreover $I\subseteq \rm{idl}(I)$. If we assume that $n=1$, one 
	can readily check that $T_{\ua}$ is a right linear map over the subring 
	given by $C^{(\si, \underline{\de)}}(\ua):=\{x\in A \mid T_{\ua}(x)=\ua x\}$. 
	In the case when $A=K$ is a division ring, $C^{(\si, \underline{\de})}(a)$ 
	is a division ring isomorphic to $\mathrm{End}_S(S/I)$, where $I=\sum_i S(t_i-a_i)$.  
	
	\begin{proposition} \label{Pro for centralizer}
		
		\begin{enumerate}
			\item[(1)] $b\in C^{(\sigma,\delta)}(\ua)$ if and only if for any $1\le i \le n$, we have $\sum_{j=1}^n\sigma_{ij}(b)a_j+\delta_i(b)-a_ib=0$.
			\item[(2)] If $I=\sum S(t_i-a_i)$, then $S/I$ is a left $S$ module and a right $C^{(\sigma,\delta)}(\ua)$- module.
			\item[(3)] There is a ring isomorphism between $C^{(\sigma,\delta)}(\ua)$ and $End_S(S/I)$, where $S=A[\ut;\sigma,\ud]$ and $I=\sum_{i=1}^nS(t_i-a_i)$.  If the base ring $A$ is a division ring these rings are in fact division rings.
			\item[(4)] We have isomorphisms of rings 
			$$
			C^{(\sigma,\delta)}(\ua)\cong End_S(S/I) \cong 
			\rm{idl}(I)/I.
			$$
 		\end{enumerate}
		
	\end{proposition}
	\begin{proof}
		(1) This is a direct consequence of the definition (cf. equation (\ref{centralizer})).
		
		(2) The fact that $S/I$ is a right $C^{(\sigma,\delta)}(\ua)$-module is clear since 
		for any $1\le i \le n$ and any $b\in C^{(\sigma,\delta)}(\ua)$, we have $(t_i-a_i)b=t_ib-a_ib=\sum_{j=1}^n\sigma_{ij}(b)t_j+\delta_i(b)-a_ib=\sum_{j=1}^n\sigma_{ij}(b)(t_j-a_j)+\sum_{j=1}^n\sigma_{ij}(b)a_j+\delta_i(b)-a_ib$.  Hence, by (1) we get $(t_i-a_i)b = \sum_{j=1}^n\sigma_{ij}(b)a_j\in I$.
		
		(3) For $b\in C^{(\sigma,\delta)}(\ua)$,  we define a map $\psi(b):S/I \rightarrow S/I$ by $\psi(b)(f(t)+I)=f(t)b+I$.  This map is well defined since, for any $1\le i \le n$ and any $s\in S$, we have (in $S/I$) $\psi(b)(s(t_i-a_i))=s(t_i-a_i)b=
		%	s(t_ib-a_ib)=s(\sum_{j=1}^n\sigma_{ij}(b)t_j+\delta_i(b)-a_ib)=s(\sum_{j=1}^n\sigma_{ij}(b)(t_j-a_j)+\sum_{j=1}^n\sigma_{ij}(b)a_j+\delta_i(b)-a_ib)=
		s(\sum_{j=1}^n\sigma_{ij}(b)(t_j-a_j)$, where the last equality is obtained as in (2) above.  The map $\psi(b)$ is easily seen to be left $S$-linear. The fact that $\psi$ is an isomorphism of rings is easy to check.  In case $A$ is a division ring, one can check that if $b\in C^{(\sigma,\delta)}(\ua)$ then $b^{-1}\in C^{(\sigma,\delta)}(\ua)$. 
		
		(4) The first isomorphism is given in (3) and the second is easy and well-known. 
	\end{proof}
	
	\begin{remark}
		{\rm 
			There is a more general point of view:  Having a left $S$ module V.   We put $C=End_S(V)$.   We then obtain a ($S,C$) bimodule structure on $V$.  
			%If such that $V$ is free and finite dimension say $n$  over $A$, $End_S(V)$ is related to some $n\times n$ matrices over $A$ acting on the right....??
			If we fix $\ua \in A^n$, and consider $V=S/I$ where $I=\sum S(t_i-a_i)$, we get a $(S,C(\ua))$ bimodule structure on $S/I$. This shed some light on the fact that $T_{\ua}$ is a $C(\ua)$ is a right module map.
		}
	\end{remark}

	\begin{proposition} \label{Pro T_a is right C linear}
		Let $\ua \in A^n$ then for any $1\le i \le n$, we have 
		\begin{center}
			$T_{a_i}\in End(A_{C})$, where $C=C^{(\sigma,\delta)}(\ua)$.
		\end{center}
	\end{proposition}
	\begin{proof}
		We have, for $x\in A$ and $y\in C $, $ T_{a_{i}}(xy) = \sum_j (\sigma_{ij}(xy) a_{j} + \delta_i(xy)) =  \sum_j (\sigma_{ij}(x) \sigma_{ij}(y) a_{j} + \sigma_{ij}(x)\delta_j(y) + \delta_i(x)y) = 
		\sum_j (\sigma_{ij}(x)( \sigma_{ij}(y) a_{j} + \delta_j(y)) + \delta_i(x)y) = \sum_j (\sigma_{ij}(x) a_iy + \delta_i(x)y)  =  \sum_j T_{a_i} (x) y.  $
	\end{proof}

	For a domain $A$, $f\in 
	S=A[\ut;\sigma,\ud]$, and $\ua \in A^n$, we define 
	$$V(f)=\{\ua \in A^n \mid f(\ua)=0 \} \; {\rm and} 
	$$
	The next proposition will put some structure on the set of roots of a polynomial $f\in S=A[\ut;\sigma,\ud]$.  The set $V(f)$ is naturally divided into conjugacy classes. For any $\ua \in V(f)$, we consider the set
	$$
	A_{\ua}:=\{x\in A \mid \exists \ub \in A^n \; {\rm with}\;  \ub x=\sigma(x )\ua + \ud(x) \} $$
	
	Since $A$ is a domain we notice that if $x\in A_{\ua}$  there exists a {\it unique} $\ub\in A^n$ such that
	$ \ub x=\sigma(x )\ua + \ud(x)$.   We will denote this unique $\ub$ by $\ua^x$.  We put
	$$E(f,\ua):=\{x \in A_{\ua} \mid 
%	\exists \ub \in A^n \;{\rm with}\; \ub x=\sigma(x )\ua + \ud(x) \;{\rm and}\; 
	f(\ua^x)=0\}
	$$
We recall that $
\Delta (\ua)=\{\ub \in A^n \mid \ua\sim \ub \}=\{\ua^x \mid x\in A\}.
$   The next proposition extends results obtained for one variable in \cite{L2012} to the multivariable setting.  Some of these have been recently obtained in the case when the coefficients are from a division ring (cf. Theorem 5 in \cite{MP2022}).

	\begin{proposition} \label{Pro for domain}
		Let $A$ be a domain, $\ua \in A^n$, and $f(\ut)\in S=A[\ut;\sigma,\ud]$. Then
		\begin{enumerate}
			\item If $0\ne x\in A$ is such that  $\ub x=\sigma(x)\ua + \ud(x)$ then 
			$x\in \ker f(T_{\ua}) $ if and only if $f(\ub)=0$
			\item $E(f,\ua)=\ker f(T_{\ua}) \cap A_{\ua}$
			\item $\ker f(T_{\ua})$ and $E(f,\ua)$ are right $C^{(\sigma,\delta)}(\ua)$ modules.
%			\item The roots of all polynomials $f(\ut)x$, $x \in A$, is a
%			union of right modules over different centralizers.  In case when $A$ is a division rings, they are contained in the set of roots of $f(\ut)$.
			\item $\Delta(\ua)\cap V(f)=\{\ua^x \mid x\in E(f,\ua)\}=\ua^{E(f,\ua)}$.
			\item Let $\Gamma =\{\ua \in A^n \mid V(f)\cap \Delta(\ua)\ne \emptyset \}$.  Then $V(f)=\bigcup_{\ua \in \Gamma} (\ua^{E(f,\ua)})$.       
		\end{enumerate}
	\end{proposition}
	\begin{proof}
		1.  The fact that $x\in \ker f(T_{\ua})$ implies $f(\ub)=0$ is given in Proposition (\ref{Pro Roots and conjugation}).  Conversely if $f(\ub)=f(\ua^x)=0$, we have 
		$0=f(\ua^x)x=fx(\ua)=f(T_{\ua})(x)$. 
		
		2.  This is clear from 1; above.
		
		3. From Proposition (\ref{Pro T_a is right C linear}), it is clear that $\ker f(T_{\ua})$ is $C$-linear.  Now let $x\in A_{\ua}$ and $\ub \in A^n$ be such that $\ub x=\sigma(x)\ua +\ud(x)$.  Let also $c\in C^{(\sigma,\delta)}(\ua)$, then $\sigma(c)\ua+\delta(c)=\ua c$ and we have $\ub xc=\sigma(x)\ua c + \ud (x)c=
		\sigma(x)(\sigma(c)\ua 
		+\ud(c))+\ud(x)c=\sigma(xc)\ua+\ud(xc)$.  This 
		shows that $xc\in A_{\ua}$ and hence $A_{\ua}$ 
		is a right $C^{(\sigma,\delta)}(\ua)$ module.  This yields the proof. 
		
		4. and 5. are clear.
	\end{proof}

	%Definition of the $\sigma,\ud$ centralizer 
	%$C^{\sigma,\ud}(\ua)$ of an element 
	%$\ua\in A^n$.	

\end{document}